\documentclass[12pt]{article}%
\usepackage{amsmath, amsthm}
\usepackage{amscd}
\usepackage{graphicx}
\usepackage{amssymb}%
\usepackage{amsmath}%
\usepackage{amsfonts}
\providecommand{\U}[1]{\protect\rule{.1in}{.1in}}
\newtheoremstyle{theorem}
{10pt}		
{10pt}
{\sl}
{\parindent}
{\bf}
{. }
{ }
{}
\providecommand{\U}[1]{\protect\rule{.1in}{.1in}}
\providecommand{\U}[1]{\protect\rule{.1in}{.1in}}
\newtheorem{theorem}{Theorem}

\newtheorem{corollary}[theorem]{Corollary}

\newtheorem{definition}[theorem]{Definition}
\newtheorem{example}[theorem]{Example}

\newtheorem{lemma}[theorem]{Lemma}

\newtheorem{proposition}[theorem]{Proposition}

\begin{document}

\title{Linear Independence of a Finite Set of Dilations by a One-Parameter Matrix Lie Group}
\author{David Ferrone$^{1}$ and Vignon Oussa$^{2}$\\$^{1}$Dept.\ of Mathematics\\University of Connecticut\\Storrs-Mansfield, CT 06269 U.S.A.\\david.ferrone@uconn.edu\\[2pt] $^{2}$Dept.\ of Mathematics \& Computer Science\\Bridgewater State University\\Bridgewater, MA 02325 U.S.A.\\vignon.oussa@bridgew.edu}
\date{}
\maketitle

\begin{abstract}
Let $G=\{e^{tA}:t\in\mathbb{R}\}$ be a closed one-parameter subgroup of the
general linear group of matrices of order $n$ acting on $\mathbb{R}^{n}$ by
matrix-vector multiplication. We assume that all eigenvalues of $A$ are
rationally related. We study conditions for which the set $\left\{  f\left(
e^{t_{1}A}\cdot\right)  ,\cdots,f\left(  e^{t_{m}A}\cdot\right)  \right\}  $
is linearly dependent in $L^{p}\left(
%TCIMACRO{\U{211d} }%
%BeginExpansion
\mathbb{R}
%EndExpansion
^{n}\right)  $ with $1\leq p<\infty.$



\end{abstract}

\section{Introduction}

In their fundamental paper \cite{Ros}, Edgar and Rosenblatt studied difference
equations over locally compact abelian groups. They define a homogeneous
linear difference equation with constant coefficients over $%
%TCIMACRO{\U{211d} }%
%BeginExpansion
\mathbb{R}
%EndExpansion
^{n}$ to be an equation of the type $\sum_{k=1}^{m}c_{k}f\left(
x+t_{j}\right)  =0$ which holds for all $x\in%
%TCIMACRO{\U{211d} }%
%BeginExpansion
\mathbb{R}
%EndExpansion
^{n}$ where $c_{1},\cdots,c_{n}$ are nonzero complex scalars, $t_{1}%
,\cdots,t_{n}$ are distinct elements in $%
%TCIMACRO{\U{211d} }%
%BeginExpansion
\mathbb{R}
%EndExpansion
^{n},$ and $f$ is a complex-valued function defined on $%
%TCIMACRO{\U{211d} }%
%BeginExpansion
\mathbb{R}
%EndExpansion
^{n}.$ A function has linearly dependent translates when it is a solution to
the difference equation. Edgar and Rosenblatt proved in \cite{Ros} that in
$L^{p}\left(
%TCIMACRO{\U{211d} }%
%BeginExpansion
\mathbb{R}
%EndExpansion
\right)  ,$ $1\leq p\leq2$ such an equation does not have a non-trivial
solution. In other words, any finite set of translates of a non-zero function
is linearly independent. However, for more general Euclidean spaces,
surprisingly, if $p>2n/\left(  n-1\right)  ,$ it is possible to find
non-trivial solutions for the difference equation. Very recently, during a
discussion related to Edgar and Rosenblatt's paper, Speegle asked the
following question. If the translations in the homogeneous linear difference
equation are replaced by dilations do the same results hold? Inspired by his
question, we consider equations of the type $\sum_{k=1}^{m}c_{k}f\left(
M_{k}x\right)  =0,$ where $\{M_{1},\cdots,M_{m}\}$ is a finite set of
non-singular matrices which are elements of a one-parameter dilation group.
More precisely, let $A$ be a matrix of order $n,$ and let $e^{A}=\sum
_{k=0}^{\infty}\frac{A^{k}}{k!}$ be an $n\times n$ non-singular matrix.
Throughout the paper, we assume that all eigenvalues of $A$ contained in $i%
%TCIMACRO{\U{211d} }%
%BeginExpansion
\mathbb{R}
%EndExpansion
$ are rationally related. That is, if $\left\{  i\beta_{1},\cdots,i\beta
_{s}\right\}  $ is the set of all eigenvalues of $A$ in $i%
%TCIMACRO{\U{211d} }%
%BeginExpansion
\mathbb{R}
%EndExpansion
^{\ast},$ then $\dim_{%
%TCIMACRO{\U{211a} }%
%BeginExpansion
\mathbb{Q}
%EndExpansion
}\left\{  i\beta_{1},\cdots,i\beta_{s}\right\}  =1.$ Let $G=\left\{
e^{tA}:t\in%
%TCIMACRO{\U{211d} }%
%BeginExpansion
\mathbb{R}
%EndExpansion
\right\}  $ be a one-parameter analytic closed subgroup of the general linear
group $GL\left(  n,%
%TCIMACRO{\U{211d} }%
%BeginExpansion
\mathbb{R}
%EndExpansion
\right)  $. Thus $G$ is a matrix Lie group. Given $v\in%
%TCIMACRO{\U{211d} }%
%BeginExpansion
\mathbb{R}
%EndExpansion
^{n},$ the group $G$ acts on $%
%TCIMACRO{\U{211d} }%
%BeginExpansion
\mathbb{R}
%EndExpansion
^{n}$ as follows. $e^{tA}\cdot v=e^{tA}v$ for $v\in%
%TCIMACRO{\U{211d} }%
%BeginExpansion
\mathbb{R}
%EndExpansion
^{n}.$ In fact, this action is just the left action given by matrix
multiplications. We are interested in the following basic question. Given a
matrix $A,$ when is true that
\begin{equation}%
%TCIMACRO{\dsum \limits_{k=1}^{m}}%
%BeginExpansion
{\displaystyle\sum\limits_{k=1}^{m}}
%EndExpansion
c_{k}\phi\left(  e^{t_{k}A}\cdot\right)  =0\label{dep}%
\end{equation}
has a non-trivial solution in $L^{p}\left(
%TCIMACRO{\U{211d} }%
%BeginExpansion
\mathbb{R}
%EndExpansion
^{n}\right)  $ with $1\leq p<\infty?$ In other words, for which function
$\phi$ is the set $\left\{  \phi\left(  e^{t_{1}A}\cdot\right)  ,\cdots
,\phi\left(  e^{t_{m}A}\cdot\right)  \right\}  $ linearly dependent in
$L^{p}\left(
%TCIMACRO{\U{211d} }%
%BeginExpansion
\mathbb{R}
%EndExpansion
^{n}\right)  $ for $1\leq p<\infty$? The unknowns in this equation are the
vectors $\left[  c_{1},c_{2},\cdots,c_{m}\right]  ,$ $\left[  t_{1}%
,t_{2},\cdots,t_{m}\right]  \in%
%TCIMACRO{\U{211d} }%
%BeginExpansion
\mathbb{R}
%EndExpansion
^{m},$ and the function $\phi.$ We assume that $t_{1},t_{2},\cdots,t_{m}$ are
distinct real numbers.

Let \textrm{supp}$\left(  \phi\right)  $ be the support of $\phi\in
L^{p}\left(
%TCIMACRO{\U{211d} }%
%BeginExpansion
\mathbb{R}
%EndExpansion
^{n}\right)  .$ Let $K\left(
%TCIMACRO{\U{211d} }%
%BeginExpansion
\mathbb{R}
%EndExpansion
\right)  $ be the set of all measurable functions of compact support on the
real line, and let $C_{0}\left(
%TCIMACRO{\U{211d} }%
%BeginExpansion
\mathbb{R}
%EndExpansion
\right)  $ be the set of all continuous functions vanishing at infinity. We
show that there exists an open dense, co-null (with respect to the Lebesgue
measure) subset of $%
%TCIMACRO{\U{211d} }%
%BeginExpansion
\mathbb{R}
%EndExpansion
^{n}$ which is $G$-invariant, and only contains locally compact $G$-orbits of
maximal dimension$.$ We call such subset $\Omega$. For each $v\in$
\textrm{supp}$\left(  \phi\right)  \cap\Omega,$ there exists a smooth
structure on the $G$-orbit of $v$ such that the map $\varphi_{v}:Gv\rightarrow%
%TCIMACRO{\U{211d} }%
%BeginExpansion
\mathbb{R}
%EndExpansion
$ defined by $\varphi_{v}\left(  e^{tA}v\right)  =t$ is a diffeomorphism. Let%
\[
\Lambda_{\phi}=\left\{  v\in\mathrm{supp}\left(  \phi\right)  \cap\Omega
:\phi_{v}\circ\varphi_{v}^{-1}\in C_{0}\left(
%TCIMACRO{\U{211d} }%
%BeginExpansion
\mathbb{R}
%EndExpansion
\right)  \text{ or }L^{p}\left(
%TCIMACRO{\U{211d} }%
%BeginExpansion
\mathbb{R}
%EndExpansion
\right)  \text{ or }K\left(
%TCIMACRO{\U{211d} }%
%BeginExpansion
\mathbb{R}
%EndExpansion
\right)  \right\}  .
\]
and let $\mathbf{m}$ be the canonical Lebesgue measure on $%
%TCIMACRO{\U{211d} }%
%BeginExpansion
\mathbb{R}
%EndExpansion
^{n}.$ We summarize our main results in the following theorem.

\begin{theorem}
\label{Theorem1}If $G$ is not the identity group, $A$ is traceless, and $G$
acts freely in $\Omega$ then the equation $\sum_{k=1}^{m}c_{k}\phi\left(
e^{t_{k}A}\cdot\right)  =0$ has no non-trivial solution in $L^{p}\left(
%TCIMACRO{\U{211d} }%
%BeginExpansion
\mathbb{R}
%EndExpansion
^{n}\right)  .$ If there exists some $v\in%
%TCIMACRO{\U{211d} }%
%BeginExpansion
\mathbb{R}
%EndExpansion
^{n}$ such that the isotropy group of $v$ is isomorphic to a lattice subgroup
of $%
%TCIMACRO{\U{211d} }%
%BeginExpansion
\mathbb{R}
%EndExpansion
$ then there is a non-trivial solution for $\sum_{k=1}^{m}c_{k}\phi\left(
e^{t_{k}A}\cdot\right)  =0.$ Suppose that $G$ is not the identity group, and
$G$ acts freely in $\Omega.$ Let $\phi\in L^{p}\left(
%TCIMACRO{\U{211d} }%
%BeginExpansion
\mathbb{R}
%EndExpansion
^{n}\right)  -0$. If $\left\{  \phi\left(  e^{t_{1}A}\cdot\right)
,\cdots,\phi\left(  e^{t_{m}A}\cdot\right)  \right\}  $ is linearly dependent
then $\mathbf{m}\left(  \Lambda_{\phi}\right)  =0$ and $m>2.$
\end{theorem}

If $\phi$ is non-trivial, a direct consequence of Theorem \ref{Theorem1} is
that, if $G$ acts freely a.e. and if $\left\{  \phi\left(  e^{t_{1}A}%
\cdot\right)  ,\cdots,\phi\left(  e^{t_{m}A}\cdot\right)  \right\} $ is
linearly dependent then $\phi$ is not continuous.

\section{Orbital Structure and Cross-sections}

Let us start by setting up some notation. Throughout the paper we write the
identity matrix of order $n$ as $1_{n\times n}.$ For the transpose of a vector
$v$ we write $v^{tr},$ and all sets are assumed to be measurable. Also, it is
worth to make the following remarks. First of all, the case where $p=\infty$
is not interesting. Also, suppose that Equation (\ref{dep}) has a non-trivial
solution, and $\phi$ is a solution such that $\phi\left(  0\right)  \neq0.$
Then it must be case that $\sum_{k=1}^{m}c_{k}=0.$ Secondly, let $\phi$ $\in
L^{1}\left(
%TCIMACRO{\U{211d} }%
%BeginExpansion
\mathbb{R}
%EndExpansion
^{n}\right)  $ be a positive function. By integrating both sides of Equation
(\ref{dep}), and by performing some change of variables on each term of the
sum, it is easy to see that if $\phi$ is a solution of Equation (\ref{dep})
then $\sum_{k=1}^{m}c_{k}\left\vert \det\left(  e^{-t_{k}A}\right)
\right\vert =0.$ Also, if the vector $\left[  t_{1},t_{2},\cdots,t_{m}\right]
$ and a function $\phi$ are given, then there is a rather simple procedure
available to decide if $\left\{  \phi\left(  e^{t_{1}A}\cdot\right)
,\cdots,\phi\left(  e^{t_{m}A}\cdot\right)  \right\}  $ is linearly independent.

\begin{proposition}
If there exists $u_{1},u_{2},\cdots,u_{m}$ in the domain of the function
$\phi$ such that the matrix
\[
\left[  \phi\left(  e^{t_{i}}u_{j}\right)  \right]  _{1\leq i,j\leq m}%
\]
is a non-singular matrix of order $m$, then $\left\{  \phi\left(  e^{t_{1}%
A}\cdot\right)  ,\cdots,\phi\left(  e^{t_{m}A}\cdot\right)  \right\}  $ is
linearly independent.

\begin{proof}
Suppose there exists $u_{1},u_{2},\cdots,u_{m}$ in the domain of the function
$\phi$ such that the matrix $X=\left[  \phi\left(  e^{t_{i}}u_{j}\right)
\right]  _{1\leq i,j\leq m}$ is non-singular, and $\phi$ satisfies Equation
(\ref{dep}) for some nonzero vector $c=\left[  c_{1},c_{2},\cdots
,c_{m}\right]  ^{tr}.$ We can then say that Equation (\ref{dep}) is equivalent
to $Xc=\left[  0,0,\cdots,0\right]  ^{tr}$. As a result, $c$ is a non-trivial
element in the null-space of the matrix $X.$ That would be a contradiction
since $X$ is invertible.
\end{proof}

\end{proposition}

If all eigenvalues of $A$ are real, we may assume that $A$ is given in real
Jordan canonical form. Otherwise, we may still use a real Jordan canonical
form. However, it is much more convenient to use a complexification of $%
%TCIMACRO{\U{211d} }%
%BeginExpansion
\mathbb{R}
%EndExpansion
^{n}$ and to use a matrix which is similar to $A,$ and is of the form $D+N$,
where $D$ is a diagonal matrix (with non real complex eigenvalues) and $N$ is
a proper lower diagonal nilpotent matrix . We explain how such form is
obtained. Let us suppose that the matrix $A$ given in real Jordan canonical
form has $p$ real eigenvalues, and $q$ non real complex eigenvalues. Notice
that $q$ is even. We identify $%
%TCIMACRO{\U{211d} }%
%BeginExpansion
\mathbb{R}
%EndExpansion
^{n}$ with $%
%TCIMACRO{\U{211d} }%
%BeginExpansion
\mathbb{R}
%EndExpansion
^{p}\times%
%TCIMACRO{\U{211d} }%
%BeginExpansion
\mathbb{R}
%EndExpansion
^{q}$ such that $n=p+q.$ Next, we fix bases for $%
%TCIMACRO{\U{211d} }%
%BeginExpansion
\mathbb{R}
%EndExpansion
^{p}\times%
%TCIMACRO{\U{211d} }%
%BeginExpansion
\mathbb{R}
%EndExpansion
^{q}$ and $%
%TCIMACRO{\U{2102} }%
%BeginExpansion
\mathbb{C}
%EndExpansion
^{p}\times%
%TCIMACRO{\U{2102} }%
%BeginExpansion
\mathbb{C}
%EndExpansion
^{q}$ such that the topological embedding map
\[
\iota:%
%TCIMACRO{\U{211d} }%
%BeginExpansion
\mathbb{R}
%EndExpansion
^{p}\times%
%TCIMACRO{\U{211d} }%
%BeginExpansion
\mathbb{R}
%EndExpansion
^{q}\rightarrow%
%TCIMACRO{\U{2102} }%
%BeginExpansion
\mathbb{C}
%EndExpansion
^{p}\times%
%TCIMACRO{\U{2102} }%
%BeginExpansion
\mathbb{C}
%EndExpansion
^{q}%
\]
defined by
\begin{align}
&  \iota\left(  \left[  v_{1},\cdots,v_{p},w_{1},\cdots,w_{q}\right]
^{tr}\right)  \label{i}\\
&  =\left[  v_{1},\cdots,v_{p},w_{1}+iw_{2},\cdots,iw_{q}+w_{q-1},w_{1}%
-iw_{2},\cdots,iw_{q}-w_{q-1}\right]  ^{tr}\nonumber
\end{align}
induces the following commutative diagram:
\[
\begin{CD}
\mathbb{R}^n @>\iota>> \mathbb{C}^p\times\mathbb{C}^q\\
@VVe^AV @VV\iota\circ e^A\circ\iota^{-1}V\\
\mathbb{R}^n @>\iota>> \mathbb{C}^p\times\mathbb{C}^q
\end{CD}
\]
Given a matrix $A$ in real Jordan canonical form, we identify $A$ with
$\iota\circ A\circ\iota^{-1}$ and $%
%TCIMACRO{\U{211d} }%
%BeginExpansion
\mathbb{R}
%EndExpansion
^{p}\times%
%TCIMACRO{\U{211d} }%
%BeginExpansion
\mathbb{R}
%EndExpansion
^{q}$ with $\iota\left(
%TCIMACRO{\U{211d} }%
%BeginExpansion
\mathbb{R}
%EndExpansion
^{p}\times%
%TCIMACRO{\U{211d} }%
%BeginExpansion
\mathbb{R}
%EndExpansion
^{q}\right)  \subset%
%TCIMACRO{\U{2102} }%
%BeginExpansion
\mathbb{C}
%EndExpansion
^{p}\times%
%TCIMACRO{\U{2102} }%
%BeginExpansion
\mathbb{C}
%EndExpansion
^{q}.$ The matrix $e^{\iota\circ A\circ\iota^{-1}}$ is then given in block
form with both real and complex generalized eigenvalues (see Example
\ref{example rotations}).

\begin{definition}
Let $\mathcal{X}$ $\ $be a $G$-invariant subset of $%
%TCIMACRO{\U{211d} }%
%BeginExpansion
\mathbb{R}
%EndExpansion
^{n}.$ A Borel set $C\subset\mathcal{X}$ is called a \textbf{cross-section}
for the action of $G$ in $\mathcal{X}$ if

\begin{enumerate}
\item $%
%TCIMACRO{\dbigcup \limits_{t\in\mathbb{R}}}%
%BeginExpansion
{\displaystyle\bigcup\limits_{t\in\mathbb{R}}}
%EndExpansion
e^{tA}C=\mathcal{X}.$

\item $e^{t_{1}A}C\cap e^{t_{2}A}C=\emptyset$ whenever $t_{1}\neq t_{2}.$
\end{enumerate}
\end{definition}

\begin{definition}
The \textbf{isotropy} \textbf{group} of $v$ is defined as
\[
G_{v}=\left\{  e^{tA}\in G:e^{tA}v=v\right\}  .
\]

\end{definition}

\begin{definition}
The action of $G$ on a $G$-invariant subset of $%
%TCIMACRO{\U{211d} }%
%BeginExpansion
\mathbb{R}
%EndExpansion
^{n}$ denoted by $\mathcal{X}$ is said to be \textbf{free }in\textbf{
}$\mathcal{X}$\textbf{ }if for any $v$ in $\mathcal{X}$, $e^{tA}v=e^{sA}v$
implies $t=s$. Equivalently, if there exists a vector $v$ in $\mathcal{X}$
such that $e^{tA}v=v$ then $t=0.$
\end{definition}

\begin{lemma}
\label{lattice}If $A$ is diagonalizable with eigenvalues in $i%
%TCIMACRO{\U{211d} }%
%BeginExpansion
\mathbb{R}
%EndExpansion
,$ there is a $G$-invariant open dense and co-null subset of $%
%TCIMACRO{\U{211d} }%
%BeginExpansion
\mathbb{R}
%EndExpansion
^{n},$ denoted $\Omega$ which only contains orbits of maximal dimension such
that the map $v\mapsto G_{v}$ is constant on $\Omega.$
\end{lemma}

\begin{proof}
For the case where $G$ is the identity group, the proof is trivial. For the
second part, suppose that $G$ is not equal to the identity group. We decompose
$%
%TCIMACRO{\U{2102} }%
%BeginExpansion
\mathbb{C}
%EndExpansion
^{p}\times%
%TCIMACRO{\U{2102} }%
%BeginExpansion
\mathbb{C}
%EndExpansion
^{q}$ into direct sums of eigensubspaces such that $%
%TCIMACRO{\U{2102} }%
%BeginExpansion
\mathbb{C}
%EndExpansion
^{p}\times%
%TCIMACRO{\U{2102} }%
%BeginExpansion
\mathbb{C}
%EndExpansion
^{q}$ is identified with
\[%
%TCIMACRO{\U{2102} }%
%BeginExpansion
\mathbb{C}
%EndExpansion
^{p}\oplus\left(
%TCIMACRO{\dbigoplus \limits_{j=1}^{q/2}}%
%BeginExpansion
{\displaystyle\bigoplus\limits_{j=1}^{q/2}}
%EndExpansion
V_{j}\right)  \oplus\left(
%TCIMACRO{\dbigoplus \limits_{j=1}^{q/2}}%
%BeginExpansion
{\displaystyle\bigoplus\limits_{j=1}^{q/2}}
%EndExpansion
\overline{V_{j}}\right)
\]
and the restriction of $\iota\circ A\circ\iota^{-1}$ onto $%
%TCIMACRO{\U{2102} }%
%BeginExpansion
\mathbb{C}
%EndExpansion
^{p}$ is a zero matrix, the restriction of $\iota\circ A\circ\iota^{-1}$ onto
$V_{k}$ is equal to a matrix $B_{k},$ and the restriction of $\iota\circ
A\circ\iota^{-1}$ onto $\overline{V_{k}}$ is equal to $\overline{B}_{k},$
where $B_{k}$ is a diagonal matrix. Moreover, $\left(  B_{k}\right)
_{j,j}=i\beta_{k},$ and $\overline{B}_{k}$ is a diagonal matrix such that
$\left(  \overline{B}_{k}\right)  _{j,j}=-i\beta_{k}$. Furthermore,
$i\beta_{k}\neq0$ for $1\leq k\leq q/2$ and $i\beta_{j}\neq i\beta_{l}$ for
$j\neq l.$ We recall that it is assumed that $\dim_{%
%TCIMACRO{\U{211a} }%
%BeginExpansion
\mathbb{Q}
%EndExpansion
}\left\{  i\beta_{1},\cdots,i\beta_{q/2}\right\}  =1.$ Now let $\pi_{k}$ be
the projection of $%
%TCIMACRO{\U{2102} }%
%BeginExpansion
\mathbb{C}
%EndExpansion
^{p}\times%
%TCIMACRO{\U{2102} }%
%BeginExpansion
\mathbb{C}
%EndExpansion
^{q}$ onto $V_{k}.$ We define $\Omega=\iota^{-1}\left(  \left\{  v\in
V:\left(  \pi_{k}\left(  v\right)  \right)  _{1}\neq0,1\leq k\leq q/2\right\}
\right)  .$ Clearly, given $v\in\Omega,$ the isotropy group of $v$ is obtained
by finding all $e^{tA}\in G,t\in%
%TCIMACRO{\U{211d} }%
%BeginExpansion
\mathbb{R}
%EndExpansion
$ such that $t\beta_{1}=2\pi k_{1},\cdots,t\beta_{q/2}=2\pi k_{q/2}$ where
$k_{1},k_{2},\cdots,k_{q/2}\in%
%TCIMACRO{\U{2124} }%
%BeginExpansion
\mathbb{Z}
%EndExpansion
.$ Finally
\[
G_{v}=\{e^{2\pi rkA}:r=\mathrm{lcm}\left(  \beta_{1},\cdots\beta_{k}\right)
,k\in\mathbb{Z}\}
\]
which is a lattice subgroup of $G.$
\end{proof}

We will show how to construct cross-sections which are different from the
cross-sections given in \cite{Speegle}. The reason why we seek to obtain a
different construction is because we would like to assert that if $G$ acts
freely on at least one vector, then there is a smooth cross-section for the
action of $G$ in some open, dense, co-null, and $G$-invariant subset of $%
%TCIMACRO{\U{211d} }%
%BeginExpansion
\mathbb{R}
%EndExpansion
^{n}.$ Also, the cross-sections obtained here will allow us to make some
precise computations later in the paper. The procedure used in this paper is
very similar to a standard technique which has been extensively used by Arnal,
Currey, and Dali in \cite{Currey} to compute cross-sections for the coadjoint
action of a class of solvable Lie groups.

Let $\iota$ be as defined in (\ref{i}) and identify $%
%TCIMACRO{\U{211d} }%
%BeginExpansion
\mathbb{R}
%EndExpansion
^{n}$ with $%
%TCIMACRO{\U{211d} }%
%BeginExpansion
\mathbb{R}
%EndExpansion
^{p}\times%
%TCIMACRO{\U{211d} }%
%BeginExpansion
\mathbb{R}
%EndExpansion
^{q}.$

\begin{lemma}
\label{cross-section} Suppose that $G$ is not equal to the identity group, and
that $G$ acts freely a.e. There exists a $G$-invariant, open, dense and
co-null subset of $\iota\left(
%TCIMACRO{\U{211d} }%
%BeginExpansion
\mathbb{R}
%EndExpansion
^{p}\times%
%TCIMACRO{\U{211d} }%
%BeginExpansion
\mathbb{R}
%EndExpansion
^{q}\right)  ,$ denoted $\Omega_{%
%TCIMACRO{\U{2102} }%
%BeginExpansion
\mathbb{C}
%EndExpansion
}$ and a function $\mathbf{F}$ defined on $\Omega_{%
%TCIMACRO{\U{2102} }%
%BeginExpansion
\mathbb{C}
%EndExpansion
}$ such that the cross-section of the $G$-action in $\Omega_{%
%TCIMACRO{\U{2102} }%
%BeginExpansion
\mathbb{C}
%EndExpansion
}$ is
\[
\Sigma_{%
%TCIMACRO{\U{2102} }%
%BeginExpansion
\mathbb{C}
%EndExpansion
}=\left\{  v\in\Omega_{%
%TCIMACRO{\U{2102} }%
%BeginExpansion
\mathbb{C}
%EndExpansion
}:\mathbf{F}\left(  v\right)  =0\right\}  .
\]
Moreover, $\Omega=\iota^{-1}\left(  \Omega_{%
%TCIMACRO{\U{2102} }%
%BeginExpansion
\mathbb{C}
%EndExpansion
}\right)  $ is a $G$-invariant open dense subset of $%
%TCIMACRO{\U{211d} }%
%BeginExpansion
\mathbb{R}
%EndExpansion
^{p}\times%
%TCIMACRO{\U{211d} }%
%BeginExpansion
\mathbb{R}
%EndExpansion
^{q}$ and a cross-section $\Sigma$ for the action of $G$ in $\Omega$ is an
$\left(  n-1\right)  $-dimensional manifold given by $\Sigma=\iota^{-1}\left(
\Sigma_{%
%TCIMACRO{\U{2102} }%
%BeginExpansion
\mathbb{C}
%EndExpansion
}\right)  .$
\end{lemma}

\begin{proof}
Let $v\in%
%TCIMACRO{\U{211d} }%
%BeginExpansion
\mathbb{R}
%EndExpansion
^{n}.$ Let us suppose that $\lambda_{1},\cdots,\lambda_{n}\in%
%TCIMACRO{\U{2102} }%
%BeginExpansion
\mathbb{C}
%EndExpansion
$ are the diagonal entries for the matrix $\iota\circ A\circ\iota^{-1}$ which
is given in the form described in (\ref{i}). If the spectrum of $A$ is real,
we let $V=%
%TCIMACRO{\U{211d} }%
%BeginExpansion
\mathbb{R}
%EndExpansion
^{n}$ and in that case the map $\iota$ is just the identity map. Otherwise, we
may assume that $%
%TCIMACRO{\U{211d} }%
%BeginExpansion
\mathbb{R}
%EndExpansion
^{n}$ is identified with $V=\iota\left(
%TCIMACRO{\U{211d} }%
%BeginExpansion
\mathbb{R}
%EndExpansion
^{p}\times%
%TCIMACRO{\U{211d} }%
%BeginExpansion
\mathbb{R}
%EndExpansion
^{q}\right)  $ as shown in (\ref{i}). Let $A=D+N,$ where $D$ is a diagonal
matrix with possibly complex entries, and $N$ is a proper lower triangular
nilpotent matrix. Computing the action of $e^{tA},$ we have%
\[
e^{tA}\left[
\begin{array}
[c]{c}%
v_{1}\\
v_{2}\\
\vdots\\
v_{n}%
\end{array}
\right]  =\left[
\begin{array}
[c]{c}%
\xi_{1}\left(  t,v\right) \\
\xi_{2}\left(  t,v\right) \\
\vdots\\
\xi_{n}\left(  t,v\right)
\end{array}
\right]  =\left[
\begin{array}
[c]{c}%
e^{t\lambda_{1}}\left(  v_{1}+p_{1}\left(  t,v\right)  \right) \\
e^{t\lambda_{1}}\left(  v_{2}+p_{2}\left(  t,v\right)  \right) \\
\vdots\\
e^{t\lambda n}\left(  v_{n}+p_{n}\left(  t,v\right)  \right)
\end{array}
\right]  ,
\]
where for each $k,$ $1\leq k\leq n$ the function $p_{k}$ is a homogeneous
polynomial of degree at most $k-1$ with $p_{1}=0$. We proceed to describe how
to obtain a cross-section for the action of $G$ in some open dense,
$G$-invariant, co-null subset of $%
%TCIMACRO{\U{211d} }%
%BeginExpansion
\mathbb{R}
%EndExpansion
^{n}$ which we will describe precisely in this proof. Let
\[
\iota\circ A\circ\iota^{-1}=\left[
\begin{array}
[c]{ccc}%
B_{1} &  & \\
& \ddots & \\
&  & B_{s}%
\end{array}
\right]  ,B_{k}=\left[
\begin{array}
[c]{ccc}%
\lambda_{B_{k}} &  & \\
\ast & \ddots & \\
& \ast & \lambda_{B_{k}}%
\end{array}
\right]  ,\lambda_{B_{k}}\in%
%TCIMACRO{\U{2102} }%
%BeginExpansion
\mathbb{C}
%EndExpansion
.
\]
There are three main cases to consider. For the first case, let us assume that
there is a block $B_{j}$ such that the non-zero entry shown below is located
at the $\mathbf{k}$-th row of matrix $\iota\circ A\circ\iota^{-1}.$ Without
loss of generality, we may assume that $j=1.$
\begin{equation}
B_{1}=\left[
\begin{array}
[c]{cccc}%
0 &  &  & \\
& \ddots &  & \\
& 1 & 0 & \\
&  &  & \ddots
\end{array}
\right]
\begin{array}
[c]{c}%
\\
\leftarrow\mathbf{k}\text{-th row}\\
\\
\end{array}
. \label{B}%
\end{equation}
We define $\Omega_{%
%TCIMACRO{\U{2102} }%
%BeginExpansion
\mathbb{C}
%EndExpansion
}=\left\{  v\in V:v_{\mathbf{k-1}}\neq0\right\}  .$ To construct a
cross-section, we solve the equation $v_{\mathbf{k}}+tv_{\mathbf{k-1}}=0$ for
$t,$ and we define the function $\mathbf{F}$ over $\Omega_{%
%TCIMACRO{\U{2102} }%
%BeginExpansion
\mathbb{C}
%EndExpansion
}$ such that $\mathbf{F}\left(  v\right)  =v_{\mathbf{k}}.$ For the second
case, we assume that case $1$ does not hold, and that without loss of
generality the first block $B_{1}$ has for eigenvalues $\alpha_{1}+i\beta_{1}$
where $\alpha_{1}\neq0.$ In that case, we define $\Omega_{%
%TCIMACRO{\U{2102} }%
%BeginExpansion
\mathbb{C}
%EndExpansion
}=\left\{  v\in V:v_{1}\neq0\right\}  $. In order to construct a
cross-section, we solve the equation $e^{t\left(  \alpha_{1}+i\beta
_{1}\right)  }v_{1}=\frac{v_{1}}{\left\vert v_{1}\right\vert },$ and we set
$\mathbf{F}\left(  v\right)  =\left\vert v_{1}\right\vert -1.$ For the last
case, we assume that case $2$ does not hold. In other words, the spectrum of
$A$ is contained in $i%
%TCIMACRO{\U{211d} }%
%BeginExpansion
\mathbb{R}
%EndExpansion
^{\ast}$ and $A$ is not diagonalizable. Let $\beta_{1}\in i%
%TCIMACRO{\U{211d} }%
%BeginExpansion
\mathbb{R}
%EndExpansion
^{\ast}.$ Without loss of generality we assume that
\[
B_{1}=\left[
\begin{array}
[c]{cccccc}%
\beta_{1} &  &  &  &  & \\
0 & \ddots &  &  &  & \\
& \ddots & \beta_{1} &  &  & \\
&  & 1 & \beta_{1} &  & \\
&  &  & \ast & \ddots & \\
&  &  &  & \ddots & \beta_{1}%
\end{array}
\right]
\begin{array}
[c]{c}%
\\
\\
\leftarrow\mathbf{k}\text{-th row}\\
\\
\\
\end{array}
.
\]
Let us also assume that the non-zero entry is located at the $\mathbf{k}$-th
row of $\iota\circ A\circ\iota^{-1}$ as shown above$.$ To construct a
cross-section, we solve the equation $\operatorname{Re}\left(  \left(
v_{\mathbf{k}}+tv_{\mathbf{k}-1}\right)  \overline{v_{\mathbf{k}-1}}\right)
=0$ for $t$ and we obtain $t=-\frac{\operatorname{Re}\left(  v_{\mathbf{k}%
}\overline{v_{\mathbf{k}-1}}\right)  }{\left\vert v_{\mathbf{k}-1}\right\vert
^{2}}.$ Now, let $t_{v}$ be the unique solution to each of the equation solved
in either case $1$, case $2$ or case $3$. The cross-section mapping
$\sigma:\Omega_{%
%TCIMACRO{\U{2102} }%
%BeginExpansion
\mathbb{C}
%EndExpansion
}\rightarrow\Sigma_{%
%TCIMACRO{\U{2102} }%
%BeginExpansion
\mathbb{C}
%EndExpansion
},$ is defined as $\sigma\left(  v\right)  =e^{\left(  t_{v}A\right)  }v.$ For
the third case, $\Omega_{%
%TCIMACRO{\U{2102} }%
%BeginExpansion
\mathbb{C}
%EndExpansion
}=\left\{  v\in V:v_{\mathbf{k}-1}\neq0\right\}  ,$ and $\mathbf{F}\left(
v\right)  =\operatorname{Re}\left(  v_{\mathbf{k}}\overline{v_{\mathbf{k}-1}%
}\right)  .$ Finally, the cross-section for the action of $G$ in $\Omega_{%
%TCIMACRO{\U{2102} }%
%BeginExpansion
\mathbb{C}
%EndExpansion
}$ is given by
\[
\Sigma_{%
%TCIMACRO{\U{2102} }%
%BeginExpansion
\mathbb{C}
%EndExpansion
}=\left\{  v\in\Omega_{%
%TCIMACRO{\U{2102} }%
%BeginExpansion
\mathbb{C}
%EndExpansion
}:\mathbf{F}\left(  v\right)  =0\right\}  .
\]
For the first case,
\[
\Sigma_{%
%TCIMACRO{\U{2102} }%
%BeginExpansion
\mathbb{C}
%EndExpansion
}=\left\{  \left[  v_{1},\cdots,v_{\mathbf{k}-1},0,v_{\mathbf{k}+1}%
,\cdots,v_{n}\right]  ^{tr}\in V:v_{\mathbf{k}-1}\neq0\right\}  .
\]
For the second case, $\Sigma_{%
%TCIMACRO{\U{2102} }%
%BeginExpansion
\mathbb{C}
%EndExpansion
}=\left\{  \left[  v_{1},\cdots,v_{n}\right]  ^{tr}\in V:v_{1}\neq0,\left\vert
v_{1}\right\vert =1\right\}  .$ Finally, for the third case,
\[
\Sigma_{%
%TCIMACRO{\U{2102} }%
%BeginExpansion
\mathbb{C}
%EndExpansion
}=\left\{  \left[  v_{1},\cdots,v_{n}\right]  ^{tr}\in V:v_{\mathbf{k}-1}%
\neq0,\operatorname{Re}\left(  v_{\mathbf{k}}\overline{v_{\mathbf{k}-1}%
}\right)  =0\right\}  .
\]
To complete the proof we define $\Omega=\iota^{-1}\left(  \Omega_{%
%TCIMACRO{\U{2102} }%
%BeginExpansion
\mathbb{C}
%EndExpansion
}\right)  ,$ $\Sigma=\iota^{-1}\left(  \Sigma_{%
%TCIMACRO{\U{2102} }%
%BeginExpansion
\mathbb{C}
%EndExpansion
}\right)  .$
\end{proof}

\begin{example}
\label{example rotations}Let
\[
A=\left[
\begin{array}
[c]{cccc}%
0 & 1 & 0 & 0\\
-1 & 0 & 0 & 0\\
1 & 0 & 0 & 1\\
0 & 1 & -1 & 0
\end{array}
\right]  \text{ and }\iota\circ A\circ\iota^{-1}=\left[
\begin{array}
[c]{cccc}%
i & 0 & 0 & 0\\
1 & i & 0 & 0\\
0 & 0 & -i & 0\\
0 & 0 & 1 & -i
\end{array}
\right]
\]
where
\[
\iota\left(  \left[  v_{1},v_{2},v_{3},v_{4}\right]  ^{tr}\right)  =\left[
v_{1}+iv_{2},v_{3}+iv_{4},v_{1}-iv_{2},v_{3}-iv_{4}\right]  ^{tr}.
\]
We have $\Omega_{%
%TCIMACRO{\U{2102} }%
%BeginExpansion
\mathbb{C}
%EndExpansion
}=\left\{  v\in\iota\left(
%TCIMACRO{\U{211d} }%
%BeginExpansion
\mathbb{R}
%EndExpansion
^{4}\right)  :v_{1}+iv_{2}\neq0\right\}  $,
\[
\Sigma_{%
%TCIMACRO{\U{2102} }%
%BeginExpansion
\mathbb{C}
%EndExpansion
}=\left\{
\begin{array}
[c]{c}%
\left[  v_{1}+iv_{2},v_{3}+iv_{4},v_{1}-iv_{2},v_{3}-iv_{4}\right]  ^{tr}\in%
%TCIMACRO{\U{2102} }%
%BeginExpansion
\mathbb{C}
%EndExpansion
^{4}:\\
v_{1}+iv_{2}\neq0,v_{1}v_{3}+v_{2}v_{4}=0
\end{array}
\right\}  ,
\]
and $\Sigma=\iota^{-1}\left(  \Sigma_{%
%TCIMACRO{\U{2102} }%
%BeginExpansion
\mathbb{C}
%EndExpansion
}\right)  .$
\end{example}

\begin{lemma}
Let $v\in\Omega.$ A $G$-orbit of $v$ is diffeomorphic to the real line, or the
circle or is equal to $v$.

\begin{proof}
Given $A$ in the form described above. $A=D+N$ where $D$ is a diagonal matrix
and $N$ is a nilpotent matrix. Thus, $e^{tA}=e^{tD}e^{tN}.$ If $N\neq0$ then
$G_{v}=\left\{  1_{n\times n}\right\}  .$ Now suppose that $N=0$ and $D\neq0.$
If there is an eigenvalue of $A$ which has a non-trivial real part, then
$G_{v}=\left\{  1_{n\times n}\right\}  .$ Otherwise, if all eigenvalues are
purely imaginary, and rationally related then $G_{v}$ is a lattice subgroup of
$G$ (see \ref{lattice})$.$ Thus $G_{v}=e^{r%
%TCIMACRO{\U{2124} }%
%BeginExpansion
\mathbb{Z}
%EndExpansion
A},$ with $r\neq0.$ Finally, if both $N$ and $D$ are equal to zero then
$G_{v}=%
%TCIMACRO{\U{211d} }%
%BeginExpansion
\mathbb{R}
%EndExpansion
$. There is a unique smooth structure on the $G$-orbit of $v,$ and with this
structure, the orbit $Gv$ is endowed with a transitive smooth action by $G.$
Furthermore the map $d:G/G_{v}\rightarrow Gv,$ with $d\left(  e^{tA}%
G_{v}\right)  =e^{tA}v$ is an equivariant diffeomorphism in the sense that $d$
is a diffeomorphism, and $d\left(  e^{sA}e^{tA}G_{v}\right)  =e^{sA}d\left(
e^{tA}G_{v}\right)  $. In fact, $Gv$ is a homogenous $G$-manifold. If
$G_{v}=\left\{  1_{n\times n}\right\}  $ then $Gv$ is diffeomorphic to the
real line $%
%TCIMACRO{\U{211d} }%
%BeginExpansion
\mathbb{R}
%EndExpansion
$. If $G_{v}=G$ then $Gv$ is diffeomorphic to $\left\{  v\right\}  $ and if
$G_{v}$ is a lattice subgroup of $%
%TCIMACRO{\U{211d} }%
%BeginExpansion
\mathbb{R}
%EndExpansion
$ then $Gv$ is diffeomorphic to a circle. This completes the proof.
\end{proof}

\end{lemma}

We will also need the following facts.

\begin{lemma}
Let $\phi:%
%TCIMACRO{\U{211d} }%
%BeginExpansion
\mathbb{R}
%EndExpansion
\rightarrow%
%TCIMACRO{\U{2102} }%
%BeginExpansion
\mathbb{C}
%EndExpansion
$ be a measurable function. If $\phi$ is a non-zero function which vanishes
outside of a compact set, then $\phi$ has linearly independent translates. If
$\phi\in C_{0}\left(
%TCIMACRO{\U{211d} }%
%BeginExpansion
\mathbb{R}
%EndExpansion
\right)  ,$ or $L^{p}\left(
%TCIMACRO{\U{211d} }%
%BeginExpansion
\mathbb{R}
%EndExpansion
\right)  $ for $p\in\left[  1,\infty\right)  $ then $\phi$ has linearly
independent translates.
\end{lemma}

\begin{proof}
The proof is in \cite{Ros}, cor 2.11 and Theorem 1.4.
\end{proof}

\section{Results}

Throughout the remainder of the paper, we may assume that $G$ is not the
identity group. In fact if $G$ is the identity group, any function in
$L^{p}\left(
%TCIMACRO{\U{211d} }%
%BeginExpansion
\mathbb{R}
%EndExpansion
^{n}\right)  $ is obviously a solution to Equation (\ref{dep}). Let
$\mathbf{m}$ be the canonical Lebesgue measure on $%
%TCIMACRO{\U{211d} }%
%BeginExpansion
\mathbb{R}
%EndExpansion
^{n}.$ Let $\phi\in L^{p}\left(
%TCIMACRO{\U{211d} }%
%BeginExpansion
\mathbb{R}
%EndExpansion
^{n}\right)  $ with $1\leq p<\infty$ be a measurable non-zero function. Let
\textrm{supp}$\left(  \phi\right)  $ be the support of $\phi.$ Then it is
clear that $\mathbf{m}\left(  \mathrm{supp}\left(  \phi\right)  \cap
\Omega\right)  >0.$

\begin{lemma}
Suppose that $G\neq\left\{  1_{n\times n}\right\}  .$ Fix $\Omega\subset%
%TCIMACRO{\U{211d} }%
%BeginExpansion
\mathbb{R}
%EndExpansion
^{n}.$ If there exists $v\in\Omega,$ such that $G_{v}$ is non-trivial then the
equation
\[
\sum_{k=1}^{m}c_{k}f\left(  e^{t_{k}A}v\right)  =0
\]
has a non-trivial solution.
\end{lemma}

\begin{proof}
The proof is fairly easy. If there exists some vector $u\in\Omega$ such that
$G_{u}$ is non-trivial, then it must be the case that $G_{v}$ is non-trivial
for every $v\in\Omega$. Since $G$ is not equal to the identity group, the
$G$-orbit of $v$ is diffeomorphic to a circle. Let us consider the unit disk
$D_{n}=\left\{  v\in\Omega:\left\Vert v\right\Vert \leq1\right\}  .$ It is
clear that for any $t\in%
%TCIMACRO{\U{211d} }%
%BeginExpansion
\mathbb{R}
%EndExpansion
,$ $e^{tA}\left(  D_{n}\right)  \subseteq D_{n}.$ Thus, letting $\phi$ be the
indicator function of the unit disk, it is easy to see that $\phi\left(
v\right)  -\phi\left(  e^{A}v\right)  =0.$
\end{proof}

Fix $\Omega\subset%
%TCIMACRO{\U{211d} }%
%BeginExpansion
\mathbb{R}
%EndExpansion
^{n}.$ From now on, we assume that the action of $G$ is \textbf{free} in
$\Omega.$ In other words $G_{v}=\left\{  I_{n\times n}\right\}  $ for all
$v\in\Omega.$ Fix $v\in$ \textrm{supp}$\left(  \phi\right)  \cap\Omega.$ We
define the map $\varphi_{v}:Gv\rightarrow%
%TCIMACRO{\U{211d} }%
%BeginExpansion
\mathbb{R}
%EndExpansion
$ with $\varphi_{v}\left(  e^{tA}v\right)  =t.$ We observe that $\varphi_{v}$
is an analytic diffeomorphism. Given any measurable function $\beta
:Gv\rightarrow%
%TCIMACRO{\U{2102} }%
%BeginExpansion
\mathbb{C}
%EndExpansion
,$ the function $\beta\circ\varphi_{v}^{-1}$ is a well-defined function of the
real line which we identify with $\beta$ via the diffeomorphism $\varphi_{v}.$
Moreover, notice that $\varphi_{v}\circ e^{tA}\circ\varphi_{v}^{-1}=L_{t}$
and
\begin{equation}
\varphi_{v}\circ e^{tA}\circ\varphi_{v}^{-1}=L_{t}\Leftrightarrow e^{tA}%
\circ\varphi_{v}^{-1}=\varphi_{v}^{-1}\circ L_{t} \label{conj}%
\end{equation}
with $L_{t}\left(  s\right)  =s+t.$ The action of $\varphi_{v}\circ
e^{tA}\circ\varphi_{v}^{-1}$ is therefore a translation action along the
$G$-orbit of $v.$ Now, let us suppose that there is a non-trivial solution in
$L^{p}\left(
%TCIMACRO{\U{211d} }%
%BeginExpansion
\mathbb{R}
%EndExpansion
^{n}\right)  $ with $1\leq p<\infty$ for equation \ref{dep}. Let $\phi$ be a
non-trivial solution and $\phi_{v}$ the restriction of the function $\phi$
onto the $G$-orbit of $v\in\Omega.$ Then it is clear that $\sum_{k=1}^{m}%
c_{k}\phi_{v}\left(  e^{t_{k}A}\cdot\right)  =0$ for some non-trivial
$c_{k}.${}

\begin{proposition}
\label{two}Suppose that $G\neq\left\{  1_{n\times n}\right\}  .$ If there
exists $v_{0}\in%
%TCIMACRO{\U{211d} }%
%BeginExpansion
\mathbb{R}
%EndExpansion
^{n}$ such that $G_{v_{0}}$ is trivial, then $c_{1}\phi\left(  e^{t_{1}A}%
\cdot\right)  +c_{2}\phi\left(  e^{t_{2}A}\cdot\right)  =0$ has no non-trivial
solution in $L^{p}\left(
%TCIMACRO{\U{211d} }%
%BeginExpansion
\mathbb{R}
%EndExpansion
^{n}\right)  $ for $c_{1},c_{2}\in%
%TCIMACRO{\U{2102} }%
%BeginExpansion
\mathbb{C}
%EndExpansion
$ and $t_{1},t_{2}\in%
%TCIMACRO{\U{211d} }%
%BeginExpansion
\mathbb{R}
%EndExpansion
.$
\end{proposition}

\begin{proof}
First, notice that since the isotropy group of $v_{0}$ is equal to the
identity group, the matrix $A$ has at least one non-zero eigenvalue which is
not an element of $i%
%TCIMACRO{\U{211d} }%
%BeginExpansion
\mathbb{R}
%EndExpansion
^{\ast},$ or $A=D+N$ where $D$ is a diagonal matrix with eigenvalues in $i%
%TCIMACRO{\U{211d} }%
%BeginExpansion
\mathbb{R}
%EndExpansion
$ and $N$ is a non-zero lower triangular nilpotent matrix. Next, if $c_{1}%
\phi\left(  e^{t_{1}A}v\right)  +c_{2}\phi\left(  e^{t_{2}A}v\right)  =0$ then
$\phi\left(  v\right)  =-\frac{c_{2}}{c_{1}}\phi\left(  e^{\left(  t_{2}%
-t_{1}\right)  A}v\right)  .$ Without loss of generality, in order to prove
the proposition, it suffices to show that equation $\phi\left(  v\right)
=c\phi\left(  e^{tA}v\right)  $ only has a trivial solution. By computing the
$p$-norm of $\phi,$ is it easy to see that $c=\left\vert \det\left(
e^{tA}\right)  \right\vert ^{1/p}.$ Also, letting $M=e^{tA}$ for any integer
$k$, if $\phi\left(  v\right)  =\left\vert \det\left(  M\right)  \right\vert
^{1/p}\phi\left(  Mv\right)  $ then $\phi\left(  v\right)  =\left\vert
\det\left(  M\right)  \right\vert ^{k/p}\phi\left(  M^{k}v\right)  .$ Now, we
consider the singly generated group $\Gamma=\left\{  M^{k}:k\in%
%TCIMACRO{\U{2124} }%
%BeginExpansion
\mathbb{Z}
%EndExpansion
\right\}  $ which also acts on $\Omega.$ Since $A$ has at least one eigenvalue
which is not purely imaginary, then $\Gamma$ is discrete, and is a closed
subgroup of $GL\left(  n,%
%TCIMACRO{\U{211d} }%
%BeginExpansion
\mathbb{R}
%EndExpansion
\right)  .$ It is not too hard to see that given $\gamma\in\Gamma,$ if
$\gamma\neq1,$ then for each point $v\in\Omega,$ there exists an open subset
$U\subset\Omega$ such that $\gamma U\cap U$ is empty. By Corollary 2.9.12 in
\cite{Var}, $\Omega/\Gamma$ has an analytic structure such that the mapping
$\gamma\mapsto\gamma v$ is a submersion. We refer the reader to \cite{Speegle}%
, where the authors construct a cross-section for the action of $\Gamma$. Let
$\Sigma^{\circ}$ be a measurable cross-section of the action of $\Gamma$ in
$\Omega,$ and $\chi_{\Sigma^{\circ}}$ the characteristic function of the set
$\Sigma^{\circ}.$ For any $k\in%
%TCIMACRO{\U{2124} }%
%BeginExpansion
\mathbb{Z}
%EndExpansion
,$ we obtain
\[
\left\Vert \phi\chi_{\Sigma^{\circ}}\right\Vert _{L^{p}\left(
%TCIMACRO{\U{211d} }%
%BeginExpansion
\mathbb{R}
%EndExpansion
^{n}\right)  }^{p}=\int_{\Sigma^{\circ}}\left\vert \det M\right\vert
^{k}\left\vert \phi\left(  M^{k}v\right)  \right\vert ^{p}dv=\int
_{M^{-k}\left(  \Sigma^{\circ}\right)  }\left\vert \phi\left(  v\right)
\right\vert ^{p}dv.
\]
However, $\left\Vert \phi\right\Vert _{L^{p}\left(
%TCIMACRO{\U{211d} }%
%BeginExpansion
\mathbb{R}
%EndExpansion
^{n}\right)  }^{p}$ is equal to $\sum_{k\in%
%TCIMACRO{\U{2124} }%
%BeginExpansion
\mathbb{Z}
%EndExpansion
}\left\Vert \phi\chi_{M^{-k}\left(  \Sigma^{\circ}\right)  }\right\Vert
_{L^{p}\left(
%TCIMACRO{\U{211d} }%
%BeginExpansion
\mathbb{R}
%EndExpansion
^{n}\right)  }^{p},$ and the latter is also equal to $\sum_{k\in%
%TCIMACRO{\U{2124} }%
%BeginExpansion
\mathbb{Z}
%EndExpansion
}\left\Vert \phi\chi_{\Sigma^{\circ}}\right\Vert _{L^{p}\left(
%TCIMACRO{\U{211d} }%
%BeginExpansion
\mathbb{R}
%EndExpansion
^{n}\right)  }^{p}.$ If $\left\Vert \phi\chi_{\Sigma^{\circ}}\right\Vert
_{L^{p}\left(
%TCIMACRO{\U{211d} }%
%BeginExpansion
\mathbb{R}
%EndExpansion
^{n}\right)  }^{p}=0$ then $\phi=0.$ Otherwise, $\phi$ is not an element of
$L^{p}\left(
%TCIMACRO{\U{211d} }%
%BeginExpansion
\mathbb{R}
%EndExpansion
^{n}\right)  .$ This completes the proof.
\end{proof}

From now on, we make the following assumptions. First, for an equation of the
type $\sum_{k=1}^{m}c_{k}\phi\left(  e^{t_{k}A}\cdot\right)  =0,$ we will
assume that $m>2.$ Secondly, we assume that $G$ acts freely on every $G$-orbit
in $\Omega.$

\begin{lemma}
\label{restriction} Suppose that $G\neq\left\{  1_{n\times n}\right\}  .$ For
any $v\in$ $\mathrm{supp}\left(  \phi\right)  \cap\Omega,$
\[%
%TCIMACRO{\dsum \limits_{k=1}^{m}}%
%BeginExpansion
{\displaystyle\sum\limits_{k=1}^{m}}
%EndExpansion
c_{k}\phi_{v}\left(  e^{t_{k}A}\cdot\right)  =%
%TCIMACRO{\dsum \limits_{k=1}^{m}}%
%BeginExpansion
{\displaystyle\sum\limits_{k=1}^{m}}
%EndExpansion
c_{k}T_{t_{k}}\left(  \phi_{v}\circ\varphi_{v}^{-1}\right)  \left(
\cdot\right)  ,
\]
where $T_{t_{k}}:L^{p}\left(
%TCIMACRO{\U{211d} }%
%BeginExpansion
\mathbb{R}
%EndExpansion
\right)  \rightarrow L^{p}\left(
%TCIMACRO{\U{211d} }%
%BeginExpansion
\mathbb{R}
%EndExpansion
\right)  $ such that $T_{t_{k}}f\left(  x\right)  =f\left(  x+t_{k}\right)  .$
\end{lemma}

\begin{proof}
Let $w\in Gv$, using (\ref{conj}), there exists a unique $t_{w}\in%
%TCIMACRO{\U{211d} }%
%BeginExpansion
\mathbb{R}
%EndExpansion
$ such that $w=e^{t_{w}A}v.$ Thus,
\[%
%TCIMACRO{\dsum \limits_{k=1}^{m}}%
%BeginExpansion
{\displaystyle\sum\limits_{k=1}^{m}}
%EndExpansion
c_{k}\phi_{v}\left(  e^{t_{k}A}w\right)  =%
%TCIMACRO{\dsum \limits_{k=1}^{m}}%
%BeginExpansion
{\displaystyle\sum\limits_{k=1}^{m}}
%EndExpansion
c_{k}\phi_{v}\left(  e^{t_{k}A}e^{t_{w}A}v\right)  =%
%TCIMACRO{\dsum \limits_{k=1}^{m}}%
%BeginExpansion
{\displaystyle\sum\limits_{k=1}^{m}}
%EndExpansion
c_{k}\phi_{v}\left(  e^{t_{k}A}\varphi_{v}^{-1}\left(  t_{w}\right)  \right)
.
\]
Finally,
\[%
%TCIMACRO{\dsum \limits_{k=1}^{m}}%
%BeginExpansion
{\displaystyle\sum\limits_{k=1}^{m}}
%EndExpansion
c_{k}\phi_{v}\left(  e^{t_{k}A}w\right)  =%
%TCIMACRO{\dsum \limits_{k=1}^{m}}%
%BeginExpansion
{\displaystyle\sum\limits_{k=1}^{m}}
%EndExpansion
c_{k}\phi_{v}\left(  \varphi_{v}^{-1}\left(  t_{w}+t_{k}\right)  \right)  =%
%TCIMACRO{\dsum \limits_{k=1}^{m}}%
%BeginExpansion
{\displaystyle\sum\limits_{k=1}^{m}}
%EndExpansion
c_{k}T_{t_{k}}\phi_{v}\left(  \varphi_{v}^{-1}\left(  t_{w}\right)  \right)
.
\]
\end{proof}

We recall that
\begin{equation}
\Lambda_{\phi}=\left\{
\begin{array}
[c]{c}%
v\in\mathrm{supp}\left(  \phi\right)  \cap\Omega:\phi_{v}\circ\varphi_{v}%
^{-1}\in C_{0}\left(
%TCIMACRO{\U{211d} }%
%BeginExpansion
\mathbb{R}
%EndExpansion
\right)  \text{ }\\
\text{or }L^{p}\left(
%TCIMACRO{\U{211d} }%
%BeginExpansion
\mathbb{R}
%EndExpansion
\right)  \text{ or }K\left(
%TCIMACRO{\U{211d} }%
%BeginExpansion
\mathbb{R}
%EndExpansion
\right)
\end{array}
\right\}  . \label{Lambda}%
\end{equation}

\begin{proposition}
\label{freeaction}Suppose that $G\neq\left\{  1_{n\times n}\right\}  ,$ and
that $G$ acts freely in $\Omega\subset%
%TCIMACRO{\U{211d} }%
%BeginExpansion
\mathbb{R}
%EndExpansion
^{n}.$ If $\mathbf{m}\left(  \Lambda_{\phi}\right)  >0$ for some $p\in\left[
1,\infty\right)  $ then the equation
\begin{equation}%
%TCIMACRO{\dsum \limits_{k=1}^{m}}%
%BeginExpansion
{\displaystyle\sum\limits_{k=1}^{m}}
%EndExpansion
c_{k}\phi\left(  e^{t_{k}A}\cdot\right)  =0 \label{equation}%
\end{equation}
does not have a non-trivial solution in $L^{p}\left(
%TCIMACRO{\U{211d} }%
%BeginExpansion
\mathbb{R}
%EndExpansion
^{n}\right)  .$
\end{proposition}

\begin{proof}
Let us suppose that Equation (\ref{equation}) has a non-trivial solution. For
each $v$ in $\Lambda_{\phi},$ applying the results in Lemma \ref{restriction}
then the restriction of a function $\phi\in L^{p}\left(
%TCIMACRO{\U{211d} }%
%BeginExpansion
\mathbb{R}
%EndExpansion
^{n}\right)  $ (which is a solution of the equation) to the $G$-orbit of
$v_{0}\in$ \textrm{supp}$\left(  \phi\right)  \cap\Omega$ yields to the
following equation.
\begin{equation}%
%TCIMACRO{\dsum \limits_{k=1}^{m}}%
%BeginExpansion
{\displaystyle\sum\limits_{k=1}^{m}}
%EndExpansion
c_{k}\phi_{v_{0}}\left(  e^{t_{k}A}\cdot\right)  =%
%TCIMACRO{\dsum \limits_{k=1}^{m}}%
%BeginExpansion
{\displaystyle\sum\limits_{k=1}^{m}}
%EndExpansion
c_{k}T_{t_{k}}\left(  \phi_{v_{0}}\circ\varphi_{v_{0}}^{-1}\right)  \left(
\cdot\right)  =0. \label{EQ}%
\end{equation}
Let us suppose that $\mathbf{m}\left(  \Lambda_{\phi}\right)  >0$ for some
$p\in\left[  1,\infty\right)  $ and equation (\ref{equation}) has a
non-trivial solution in $L^{p}\left(
%TCIMACRO{\U{211d} }%
%BeginExpansion
\mathbb{R}
%EndExpansion
^{n}\right)  .$ Using Cor $2.11$ and Theorem $1.4$ in \cite{Ros} then
$\phi_{v}\left(  \varphi_{v}^{-1}\left(  t\right)  \right)  =0$ for all $t\in%
%TCIMACRO{\U{211d} }%
%BeginExpansion
\mathbb{R}
%EndExpansion
$ and for all $v$ in $\Lambda_{\phi}.$ Let $\chi_{\Lambda_{\phi}}$ be the
characteristic function of the set $\Lambda_{\phi}.$ Thus $\phi\chi
_{\Lambda_{\phi}}$ is a zero function. That would be a contradiction, since
$\Lambda_{\phi}\subset$ \textrm{supp}$\left(  \phi\right)  .$
\end{proof}

Let $C\left(
%TCIMACRO{\U{211d} }%
%BeginExpansion
\mathbb{R}
%EndExpansion
^{n}\right)  $ be the set of all continuous function on $%
%TCIMACRO{\U{211d} }%
%BeginExpansion
\mathbb{R}
%EndExpansion
^{n}.$

\begin{corollary}
The equation $\sum_{k=1}^{m}c_{k}\phi\left(  e^{t_{k}A}\cdot\right)  =0$ has
no non-zero solution in $L^{p}\left(
%TCIMACRO{\U{211d} }%
%BeginExpansion
\mathbb{R}
%EndExpansion
^{n}\right)  \cap C\left(
%TCIMACRO{\U{211d} }%
%BeginExpansion
\mathbb{R}
%EndExpansion
^{n}\right)  .$

\begin{proof}
Suppose by contradiction that $\phi\in L^{p}\left(
%TCIMACRO{\U{211d} }%
%BeginExpansion
\mathbb{R}
%EndExpansion
^{n}\right)  \cap C\left(
%TCIMACRO{\U{211d} }%
%BeginExpansion
\mathbb{R}
%EndExpansion
^{n}\right)  $ is a non-trivial solution to the equation. Clearly
$\mathbf{m}\left(  \Lambda_{\phi}\right)  >0.$ According to Proposition
\ref{freeaction} that would be a contradiction.
\end{proof}

\end{corollary}

Now, we consider the diffeomorphism $\xi:%
%TCIMACRO{\U{211d} }%
%BeginExpansion
\mathbb{R}
%EndExpansion
\times\Sigma\rightarrow\Omega$ defined by $\xi\left(  t,w\right)  =e^{tA}w.$
Let $\phi$ be a function in $L^{p}\left(
%TCIMACRO{\U{211d} }%
%BeginExpansion
\mathbb{R}
%EndExpansion
^{n}\right)  ,$ and let $J_{\xi}\left(  t,w\right)  $ be the Jacobian of $\xi$
at $\left(  t,w\right)  \in%
%TCIMACRO{\U{211d} }%
%BeginExpansion
\mathbb{R}
%EndExpansion
\times\Sigma.$ Let $L\left(
%TCIMACRO{\U{211d} }%
%BeginExpansion
\mathbb{R}
%EndExpansion
\times\Sigma\right)  $ be set of all Lebesgue measurable functions on $%
%TCIMACRO{\U{211d} }%
%BeginExpansion
\mathbb{R}
%EndExpansion
\times\Sigma.$ We define the operator $\mathbf{D:}$ $L^{p}\left(
%TCIMACRO{\U{211d} }%
%BeginExpansion
\mathbb{R}
%EndExpansion
^{n}\right)  \rightarrow L\left(
%TCIMACRO{\U{211d} }%
%BeginExpansion
\mathbb{R}
%EndExpansion
\times\Sigma\right)  $ such that
\begin{equation}
\mathbf{D}\phi\left(  w,t\right)  =\phi\left(  \xi\left(  t,w\right)  \right)
\left\vert \det\left(  J_{\xi}\left(  t,w\right)  \right)  \right\vert ^{1/p}.
\label{D}%
\end{equation}
We recall the definitions of the objects $\Sigma_{%
%TCIMACRO{\U{2102} }%
%BeginExpansion
\mathbb{C}
%EndExpansion
}\ $and $\Omega_{%
%TCIMACRO{\U{2102} }%
%BeginExpansion
\mathbb{C}
%EndExpansion
}$ as given in Lemma \ref{cross-section}. Define $\xi_{%
%TCIMACRO{\U{2102} }%
%BeginExpansion
\mathbb{C}
%EndExpansion
}:%
%TCIMACRO{\U{211d} }%
%BeginExpansion
\mathbb{R}
%EndExpansion
\times\Sigma_{%
%TCIMACRO{\U{2102} }%
%BeginExpansion
\mathbb{C}
%EndExpansion
}\rightarrow\Omega_{%
%TCIMACRO{\U{2102} }%
%BeginExpansion
\mathbb{C}
%EndExpansion
}$ such that $\xi_{%
%TCIMACRO{\U{2102} }%
%BeginExpansion
\mathbb{C}
%EndExpansion
}\left(  t,w\right)  =e^{tA}w.$ Let $J_{\xi_{%
%TCIMACRO{\U{2102} }%
%BeginExpansion
\mathbb{C}
%EndExpansion
}}$ be the Jacobian of the map $\xi_{%
%TCIMACRO{\U{2102} }%
%BeginExpansion
\mathbb{C}
%EndExpansion
}.$

\begin{lemma}
\label{LemmaM} $\left\vert \det\left(  J_{\xi_{%
%TCIMACRO{\U{2102} }%
%BeginExpansion
\mathbb{C}
%EndExpansion
}}\left(  t,w\right)  \right)  \right\vert =\left\vert \det\left(
e^{tA}\right)  \right\vert \mathbf{f}_{%
%TCIMACRO{\U{2102} }%
%BeginExpansion
\mathbb{C}
%EndExpansion
}\left(  w\right)  ,$ where $\mathbf{f}_{%
%TCIMACRO{\U{2102} }%
%BeginExpansion
\mathbb{C}
%EndExpansion
}$ is some positive function defined on the cross-section $\Sigma_{%
%TCIMACRO{\U{2102} }%
%BeginExpansion
\mathbb{C}
%EndExpansion
}.$
\end{lemma}

\begin{proof}
Let $\lambda_{1},\cdots,\lambda_{n}\in%
%TCIMACRO{\U{2102} }%
%BeginExpansion
\mathbb{C}
%EndExpansion
$ be the eigenvalues of the matrix $A.$ Following the proof of Lemma
\ref{cross-section}, we have three separate cases to consider. With some
elementary computations, we obtain for case $1,$ $\left\vert \det\left(
J_{\xi_{%
%TCIMACRO{\U{2102} }%
%BeginExpansion
\mathbb{C}
%EndExpansion
}}\left(  t,w\right)  \right)  \right\vert =\left\vert \det\left(
e^{tA}\right)  \right\vert \left\vert v_{\mathbf{k}-1}\right\vert .$ For case
$2,$%
\[
J_{\xi_{%
%TCIMACRO{\U{2102} }%
%BeginExpansion
\mathbb{C}
%EndExpansion
}}\left(  t,w\right)  =\left[
\begin{array}
[c]{cccc}%
\lambda_{1}e^{t\lambda_{1}}x &  &  & \\
\ast & e^{t\lambda_{2}} &  & \\
\vdots & \ast & \ddots & \\
\ast & \cdots & \ast & e^{t\lambda_{n}}%
\end{array}
\right]  ,
\]
$\left\vert \det\left(  J_{\xi_{%
%TCIMACRO{\U{2102} }%
%BeginExpansion
\mathbb{C}
%EndExpansion
}}\left(  t,w\right)  \right)  \right\vert =\left\vert \det\left(
e^{tA}\right)  \right\vert \left\vert \lambda_{1}\right\vert ,$ and finally
for case $3$ we obtain that $\left\vert \det\left(  J_{\xi_{%
%TCIMACRO{\U{2102} }%
%BeginExpansion
\mathbb{C}
%EndExpansion
}}\left(  t,w\right)  \right)  \right\vert =\left\vert \det\left(
e^{tA}\right)  \right\vert \left\vert v_{\mathbf{k}-1}+iv_{\mathbf{k}%
}\right\vert $ where $\operatorname{Re}\left(  v_{\mathbf{k}}\overline
{v_{\mathbf{k}-1}}\right)  =0.$
\end{proof}

Applying Lemma \ref{LemmaM}, and Lemma 12.2.2 in \cite{Comp}, there exists a
positive function $\mathbf{f}$ defined on the cross-section $\Sigma$ such
that
\begin{equation}
\left\vert \det\left(  J_{\xi}\left(  t,w\right)  \right)  \right\vert
=\left\vert \det\left(  e^{tA}\right)  \right\vert \mathbf{f}\left(  w\right)
. \label{J}%
\end{equation}

\begin{lemma}
\label{Lemma above}Suppose that $G\neq\left\{  1_{n\times n}\right\}  ,$ and
that $G$ acts freely in $\Omega\subset%
%TCIMACRO{\U{211d} }%
%BeginExpansion
\mathbb{R}
%EndExpansion
^{n}.$ If $\phi\in L^{p}\left(
%TCIMACRO{\U{211d} }%
%BeginExpansion
\mathbb{R}
%EndExpansion
^{n}\right)  $ then $\left\Vert \phi\right\Vert _{L^{p}\left(
%TCIMACRO{\U{211d} }%
%BeginExpansion
\mathbb{R}
%EndExpansion
^{n}\right)  }=\left\Vert \mathbf{D}\phi\right\Vert _{L^{p}\left(
%TCIMACRO{\U{211d} }%
%BeginExpansion
\mathbb{R}
%EndExpansion
\times\Sigma\right)  }$
\end{lemma}

\begin{proof}
Clearly, we have that $\left\Vert \phi\right\Vert _{L^{p}\left(
%TCIMACRO{\U{211d} }%
%BeginExpansion
\mathbb{R}
%EndExpansion
^{n}\right)  }^{p}=\int_{\Omega}\left\vert \phi\left(  u\right)  \right\vert
^{p}du$ and
\begin{align*}
\left\Vert \phi\right\Vert _{L^{p}\left(
%TCIMACRO{\U{211d} }%
%BeginExpansion
\mathbb{R}
%EndExpansion
^{n}\right)  }^{p}  &  =\int_{\Sigma}\int_{%
%TCIMACRO{\U{211d} }%
%BeginExpansion
\mathbb{R}
%EndExpansion
}\left\vert \phi\left(  \xi\left(  t,w\right)  \right)  \right\vert
^{p}\left\vert \det\left(  J_{\xi}\left(  t,w\right)  \right)  \right\vert
dtdw\\
&  =\int_{\Sigma}\int_{%
%TCIMACRO{\U{211d} }%
%BeginExpansion
\mathbb{R}
%EndExpansion
}\left\vert \phi\left(  \xi\left(  t,w\right)  \right)  \left\vert \det\left(
J_{\xi}\left(  t,w\right)  \right)  \right\vert ^{1/p}\right\vert ^{p}dtdw\\
&  =\int_{\Sigma}\int_{%
%TCIMACRO{\U{211d} }%
%BeginExpansion
\mathbb{R}
%EndExpansion
}\left\vert \mathbf{D}\phi\left(  w,t\right)  \right\vert ^{p}dtdw\\
&  =\left\Vert \mathbf{D}\phi\right\Vert _{L^{p}\left(
%TCIMACRO{\U{211d} }%
%BeginExpansion
\mathbb{R}
%EndExpansion
\times\Sigma\right)  }.
\end{align*}
\end{proof}

\begin{lemma}
\label{Lemma1}Assume that $G\neq\left\{  1_{n\times n}\right\}  ,$ that $G$
acts freely in $\Omega\subset%
%TCIMACRO{\U{211d} }%
%BeginExpansion
\mathbb{R}
%EndExpansion
^{n},$ and $\det\left(  e^{A}\right)  =1.$ If $\phi\in$ $L^{p}\left(
%TCIMACRO{\U{211d} }%
%BeginExpansion
\mathbb{R}
%EndExpansion
^{n}\right)  $ then $\mathbf{m}\left(  \Lambda_{\phi}\right)  >0.$
\end{lemma}

\begin{proof}
If $\det e^{A}=1$ then $t\mapsto\det e^{tA}$ is constant on $%
%TCIMACRO{\U{211d} }%
%BeginExpansion
\mathbb{R}
%EndExpansion
.$ Now suppose by contradiction that $\phi\in$ $L^{p}\left(
%TCIMACRO{\U{211d} }%
%BeginExpansion
\mathbb{R}
%EndExpansion
^{n}\right)  $ and $\mathbf{m}\left(  \Lambda_{\phi}\right)  =0.$ Let $\mu$ be
the Lebesgue measure on $\Sigma.$ Let $\Sigma_{\phi}=\Sigma\cap\Lambda_{\phi
}.$ We observe that if $\mathbf{m}\left(  \Lambda_{\phi}\right)  =0$ then
$\mu\left(  \Sigma_{\phi}\right)  =0.$ Applying Lemma \ref{Lemma above}, and
the fact that $\phi\left(  \xi\left(  t,w\right)  \right)  =\phi_{w}\left(
\varphi_{w}^{-1}\left(  t\right)  \right)  ,$
\begin{align*}
\left\Vert \phi\right\Vert _{L^{p}\left(
%TCIMACRO{\U{211d} }%
%BeginExpansion
\mathbb{R}
%EndExpansion
^{n}\right)  }^{p}  &  =\int_{\Sigma}\int_{%
%TCIMACRO{\U{211d} }%
%BeginExpansion
\mathbb{R}
%EndExpansion
}\left\vert \phi\left(  \xi\left(  t,w\right)  \right)  \right\vert
^{p}\left\vert \mathbf{f}\left(  w\right)  \right\vert dtdw\\
&  =\int_{\Sigma}\left\vert \mathbf{f}\left(  w\right)  \right\vert \left(
\left\Vert \phi_{w}\circ\varphi_{w}^{-1}\right\Vert _{L^{p}\left(
%TCIMACRO{\U{211d} }%
%BeginExpansion
\mathbb{R}
%EndExpansion
\right)  }^{p}\right)  dw.
\end{align*}
However,
\begin{align*}
\left\Vert \phi\right\Vert _{L^{p}\left(
%TCIMACRO{\U{211d} }%
%BeginExpansion
\mathbb{R}
%EndExpansion
^{n}\right)  }^{p}  &  =\left(  \int_{\Sigma_{\phi}}\left\vert \mathbf{f}%
\left(  w\right)  \right\vert \left(  \left\Vert \phi_{w}\circ\varphi_{w}%
^{-1}\right\Vert _{L^{p}\left(
%TCIMACRO{\U{211d} }%
%BeginExpansion
\mathbb{R}
%EndExpansion
\right)  }^{p}\right)  dw\right)  +\\
&  \left(  \int_{\Sigma-\Sigma_{\phi}}\left\vert \mathbf{f}\left(  w\right)
\right\vert \left(  \left\Vert \phi_{w}\circ\varphi_{w}^{-1}\right\Vert
_{L^{p}\left(
%TCIMACRO{\U{211d} }%
%BeginExpansion
\mathbb{R}
%EndExpansion
\right)  }^{p}\right)  dw\right)  .
\end{align*}
Since $\Sigma_{\phi}$ is a $\mu$-null set.
\[
\left\Vert \phi\right\Vert _{L^{p}\left(
%TCIMACRO{\U{211d} }%
%BeginExpansion
\mathbb{R}
%EndExpansion
^{n}\right)  }^{p}=\int_{\Sigma-\Sigma_{\phi}}\left\vert \mathbf{f}\left(
w\right)  \right\vert \left(  \left\Vert \phi_{w}\circ\varphi_{w}%
^{-1}\right\Vert _{L^{p}\left(
%TCIMACRO{\U{211d} }%
%BeginExpansion
\mathbb{R}
%EndExpansion
\right)  }^{p}\right)  dw.
\]
Thus $\left\Vert \phi\right\Vert _{L^{p}\left(
%TCIMACRO{\U{211d} }%
%BeginExpansion
\mathbb{R}
%EndExpansion
^{n}\right)  }^{p}$ is either undefined or infinite. That would be a contradiction.
\end{proof}

Recall that $\det\left(  e^{A}\right)  =\exp\left(  \mathrm{trace}\left(
A\right)  \right)  .$

\begin{proposition}
Assume that $G\neq\left\{  1_{n\times n}\right\}  ,$ and that $G$ acts freely
in $\Omega\subset%
%TCIMACRO{\U{211d} }%
%BeginExpansion
\mathbb{R}
%EndExpansion
^{n}.$ If $\mathrm{trace}\left(  A\right)  =0$ then equation $\sum_{k=1}%
^{m}c_{k}\phi\left(  e^{t_{k}A}\cdot\right)  =0$ has no non-trivial solution
in $L^{p}\left(
%TCIMACRO{\U{211d} }%
%BeginExpansion
\mathbb{R}
%EndExpansion
^{n}\right)  .$
\end{proposition}

\begin{proof}
The proof follows from Lemma \ref{Lemma1}.
\end{proof}

\begin{example}
Coming back to Example \ref{example rotations},
\[
\exp\left(  tA\right)  \left[
\begin{array}
[c]{c}%
v_{1}\\
v_{2}\\
v_{3}\\
v_{4}%
\end{array}
\right]  =\left[
\begin{array}
[c]{c}%
v_{1}\cos t+v_{2}\sin t\\
v_{2}\cos t-v_{1}\sin t\\
\left(  v_{3}+tv_{1}\right)  \cos t+\left(  tv_{2}+v_{4}\right)  \sin t\\
\left(  v_{4}+tv_{2}\right)  \cos t-\left(  v_{3}+tv_{1}\right)  \sin t
\end{array}
\right]  .
\]
Notice that $A$ is traceless. Thus equation $\sum_{k=1}^{m}c_{k}\phi\left(
e^{t_{k}A}\cdot\right)  =0$ has no non-trivial solution in $L^{p}\left(
%TCIMACRO{\U{211d} }%
%BeginExpansion
\mathbb{R}
%EndExpansion
^{4}\right)  .$
\end{example}

\end{document}